\documentclass[11 pt]{amsart}
\usepackage{amscd,amsfonts,amssymb,amsmath}
\usepackage{hyperref}
\usepackage{epsfig}
\usepackage{mathtools}
\usepackage{tabu}
\usepackage{tikz-cd}
\newtheorem{theorem}{Theorem}[section]
\newtheorem{corollary}[theorem]{Corollary}
\newtheorem{lemma}[theorem]{Lemma}
\newtheorem{proposition}[theorem]{Proposition}
\theoremstyle{definition}

\newtheorem{question}[theorem]{Question}

\newtheorem{definition}[theorem]{Definition}
\newtheorem{example}[theorem]{Example}
\newtheorem{remark}[theorem]{Remark}
\numberwithin{equation}{subsection}
\newtheorem*{ack}{Acknowledgement}
\usepackage[all,cmtip]{xy}

\usepackage{graphicx} 

\newcommand{\Alex}{\operatorname{Alex}}
\newcommand{\Maps}{\operatorname{Map}}
\newcommand{\Aut}{\operatorname{Aut}}
\newcommand{\Hol}{\operatorname{Hol}}

\newcommand{\Conj}{\operatorname{Conj}}
\newcommand{\Core}{\operatorname{Core}}

\newcommand{\Inn}{\operatorname{Inn}}

\newcommand{\T}{\operatorname{T}}

\newcommand{\R}{\operatorname{R}}

\newcommand{\id}{\mathrm{id}}

\newcommand{\cs}{\circ*\,}
\newcommand{\cc}{*\circ\,}
\newcommand{\bs}{\,\bar{*}\,}
\newcommand{\bc}{\bar{\,\circ\,}}

\begin{document}

\title{Multiplication of   quandle structures}
\author{Valeriy G. Bardakov}
\author{Denis A. Fedoseev}

\date{\today}
\address{Sobolev Institute of Mathematics and Novosibirsk State University, Novosibirsk 630090, Russia.}
\address{Novosibirsk State Agrarian University, Dobrolyubova street, 160, Novosibirsk, 630039, Russia.}
\address{Regional Scientific and Educational Mathematical Center of Tomsk State University,
36 Lenin Ave., 14, 634050, Tomsk, Russia.}
\email{bardakov@math.nsc.ru}

\address{Moscow state university}
\email{denfedex@yandex.ru}


\begin{abstract}
We generalise the construction of $Q$-family of quandles and $G$-family of quandles which were introduced in the paper of A. Ishii, M. Iwakiri,  Y. Jang, K. Oshiro, and find connection with other constructions of quandles.  We define a composition  of quandl's structures, which are defined on the same set and find conditions under which  this composition gives a quandle. Further we prove that under this multiplication we get a group and show that this group is abelian.
\end{abstract}
\maketitle

\bigskip

\section{Introduction}\label{introduction}

A quandle is an algebraic system with a single binary operation that satisfies three axioms encoding the three Reidemeister moves on diagrams of knots and links in the 3-space. 
Quandles were introduced in the fundamental works of Joyce \cite{J} and Matveev \cite{Matveev}, who showed that link quandles are complete invariants of non-split links up to orientation change of the ambient space.
These objects appear in diverse range of areas of mathematics, namely, knot theory \cite{J, Matveev, Nelson}, group theory, set-theoretic solutions to the Yang-Baxter equations and Yetter-Drinfeld Modules \cite{Eisermann2005}, Riemannian symmetric spaces \cite{Loos1}, and Hopf algebras \cite{A}.  
Although link quandles are strong invariants, it is difficult to check whether two quandles are isomorphic. 
\par

During last years quandles have been studied as algebraic system without connection with knot theory. In particular, in \cite{A, BarTimSin2019,  BarSin, E} some constructions of quandles, extensions of quandles, automorphisms of quandles, and representations of quandles are studied. A good description of the algebraic theory of quandles can be found in the book \cite{T}.
A lot of quandles can be constructed from groups. In particular, conjugation quandles, core quandles, generalised Alexander quandles are constructed on a group, and the quandle operation is defined in terms of the original group operation. But it is known that not any quandle comes from a group. In \cite{BN} some construction of quandles, which is a generalization of the construction of conjugation quandle was suggested.
Automorphisms of quandles, which reveal a lot about their internal structures, have been investigated in much detail in a series of papers \cite{Bardakov2017, BarTimSin, Elhamdadi2012}. 
In an attempt to linearise the study of quandles, a theory of quandle rings analogous to the classical theory of group rings was proposed in \cite{Bardakov2019, BPS, Elhamdadi2019}. 

In papers \cite{I1, I2} {\em $Q$-family of quandles} and {\em $G$-family of quandles} were introduced.  A $Q$-family of quandles is a non-empty set $X$ with a set of quandle operations $*_a$, which are indexed by elements of the quandle $Q$. The definition of $G$-family of quandles is similar but instead a quandle $Q$ a group $G$ is used. For $Q$-family of quandles or $G$-family of quandles the {\em associated quandles} $X \times Q$ (resp. $X \times G$) are defined.
Just like quandle axioms were motivated by moves of knot theory, the axioms of a $G$-family of quandles are motivated by handlebody-knot theory,
in which they are used to construct invariants of handlebody-knots \cite{I, I1}. A handlebody-knot is a handlebody embedded in the 3-sphere. A handlebody-knot can be represented by its trivalent spine, and in \cite{I2} a list of local moves connecting diagrams of spatial trivalent graphs which represent equivalent handlebody-knots was given.

In the present paper we find connection between $Q$-family of quandles and quandle extensions from \cite{A}. We also define $(Q, f)$-family of quandles and $(G, f)$-family of quandles which generalise constructions of $Q$-family of quandles and $G$-family of quandles.

A $G$-family of quandles gives a group structure on the set of quandles $(X, *_g)$, $g \in G$, by the rule  $(X, *_g) \circ (X, *_h) = (X, *_{gh})$. We consider more general situation. Suppose that on a set $X$ we have two quandle structures $(X, *_1)$ and  $(X, *_2)$. Then we define their {\em quandle product} $(X, *_1) \circ (X, *_2) = (X, *_1 *_2)$ that is an algebraic system (groupoid) on $X$ with the operation 
$$
x *_1 *_2 y = (x *_1 y) *_2 y,~~x, y \in X,
$$
which is called the {\em composition of quandle operations} $*_1$ and $*_2$.
In general this algebraic system is not a quandle. We prove that if $*_2$ is distributive with respect to $*_1$ (that is, for any $x, y, z \in X$ the equality
$$
(x *_1 y) *_2 z = (x *_2 z) *_1 (y *_2 z)
$$
holds), then $(X, *_1 *_2)$ is a quandle. If $*_2$ is distributive with respect to $*_1$ and vice versa, then we can define an operation $*_1^{n_1} *_2^{m_1} \ldots *_1^{n_k} *_2^{m_k}$, where $n_i$ and $m_i$ are integers, on $X$. Theorem \ref{thm:quandle_mult} says that this operation defines a quandle on $X$. It gives a group structure on this set of quandles. Theorem \ref{gent} generalizes this situation: if we have a set of quandles $(X, *_i)$, $i \in \Lambda$, on a set $X$ and for any pairs $i, j \in \Lambda$, $*_i$ is distributive with respect to $*_j$ and vice versa, then there exists a group that is generated by $(X, *_i)$ under the multiplication $(X, *_i) \circ (X, *_j) = (X, *_i *_j)$. The unit element of this group is the trivial quandle on $X$. We prove (see Corollary \ref{abel}) that this group is abelian.


If $G$ is a non-abelian group, then we can define two quandles on the set $G$: $\Conj(G)$ and $\Core(G)$. It is interesting to understand under which conditions they generate a group. We prove that if $G$ is a two-step nilpotent group in which the square of any element lies in the center, then $\Conj(G)$ and $\Core(G)$ generate a group that is isomorphic to the abelianization $G^{ab} = G / G'$.

In the end, we consider the following situation. Fix a group $G$ and take two of its automorphisms $\varphi, \psi \in \Aut(G)$. We can define two generalised Alexander quandles on the set $G$. The first one is a quandle $G^*$ with multiplication
$a * b = \varphi(a b^{-1}) b$, $a, b \in G$; the second one is a quandle $G^{\circ}$ with multiplication
$a \circ b = \psi(a b^{-1}) b$, $a, b \in G$. We prove in Proposition \ref{alex} that if $\varphi \psi = \psi \varphi$, then for any integers $k, l$ the groupoid $(G, *^k \circ^l)$ is a quandle.

The paper is organized as follows.
In Section~\ref{sec-prelim} we give the preliminary information on groupoids and quandles.
In Section~\ref{generalizations} we suggest some generalizations of $Q$-family of quandles and $G$-family of quandles.
In Section~\ref{sec:quandle_mult} we define composition of quandle operations and multiplication of quandles which are defined on the same set. In particular, we consider the conjugation quandle and the core quandle on a group $G$ and find in Corollary~\ref{conjcor} condition on $G$ under which these two quandles generate a group.
For a pair of generalised Alexander quandles on $G$ we get a similar result.

In the last section we formulate questions for further research.

\section{Preliminaries on quandles}\label{sec-prelim}

In this section we recall some well-known facts on quandles (see, for example, \cite{A, J}).

By {\it groupoid} $(Q, *)$ we mean a non-empty set  $Q$ with one binary algebraic operation $*\colon Q \times Q \to Q$.
A {\it quandle} is a groupoid  $(Q, *)$ in which the  operation $(x,y) \mapsto x * y$ satisfies the following axioms:
\begin{enumerate}
\item[(Q1)] Idempotency axiom: $x*x=x$ for all $x \in Q$,
\item[(Q2)] Right invertibility axiom: for any $x,y \in Q$ there exists a unique $z \in Q$ such that $x=z*y$,
\item[(Q3)] Self-distributivity axiom: $(x*y)*z=(x*z) * (y*z)$ for all $x,y,z \in Q$.
\end{enumerate}

A groupoid satisfying (Q1)  is called an {\it idempotent groupoid}. A groupoid satisfying (Q2) is called a {\it right quasigroup}.
A groupoid satisfying  (Q2) and (Q3)  is called a {\it rack}.

\begin{remark}
It is not difficult to see that a quandle which is generated by one element contains only one element, but a rack which is generated by one element can contain infinite number of elements.
\end{remark}

It follows from (Q2) that we can define an operation $*^{-1}\colon Q \times Q \to Q$ by the rule
$$
a = c * b \Leftrightarrow c = a *^{-1} b.
$$
This is equivalent to
$$
(a * b) *^{-1} b = a = (a *^{-1} b) * b.
$$
From this identity and (Q3) follow the identities
$$
(a * b) *^{-1} c = (a *^{-1} c) * (b *^{-1} c),
$$
$$
(a *^{-1} b) * c = (a * c) *^{-1} (b * c),
$$
$$
(a *^{-1} b) *^{-1} c = (a *^{-1} c) *^{-1} (b *^{-1} c).
$$

Hence, we can define a rack as a set $Q$ equipped with two binary operations 
$$
(a, b) \mapsto a * b~\mbox{and}~ (a, b) \mapsto a *^{-1} b
$$ 
which satisfy the axioms
\begin{enumerate}
\item[(R1)] $(a * b) *^{-1} b = (a *^{-1} b) * b = a$ for all $a, b, c \in Q$,
\item[(R2)] $(a*b)*c=(a*c) * (b*c)$ for all $a, b, c \in Q$.
\end{enumerate}

Using this observation one can prove

\begin{proposition} \label{NorFor}
Let $Q$ be a rack with a set of generators $A$. Then any element $u \in Q$ can be presented in the form
$$
u = (\ldots ((a_{i_1} *^{\varepsilon_1} a_{i_2}) *^{\varepsilon_2} a_{i_3}) *^{\varepsilon_3} \ldots  ) *^{\varepsilon_{m-1}} a_{i_m},~~~a_{i_j} \in A,~~\varepsilon_k \in \{ -1, 1 \}.
$$
\end{proposition}

Many interesting examples of quandles come from groups showing the deep connection of quandle theory with group theory.

\begin{itemize}
\item If $G$ is a group and $n$ is an integer, then the set $G$ equipped with the binary operation $a*b= b^{-n} a b^n$ forms a quandle $\Conj_n(G)$. For $n=1$, it is called the {\it conjugation quandle} $\Conj(G)$.
\item If $G$ is a group, then the binary operation $a*b= b a^{-1} b$ turns the set $G$ into the {\it core quandle}  $\Core(G)$.
\item If $\varphi \in \Aut(G)$ is an automorphism of a group $G$, then the set $G$ with the binary operation $a*b=\varphi (ab^{-1})b$ forms a quandle $\Alex_{\varphi } (G)$ called the \textit{generalised  Alexander quandle}. In particular, if $G= \mathbb{Z}/n \mathbb{Z}$ and $\varphi $ is the inversion, then we get the {\it dihedral quandle} $\R_n$.
\end{itemize}
\bigskip

A quandle $Q$ is called {\it trivial} if $x*y=x$ for all $x, y \in Q$.  Unlike groups, a trivial quandle can contain arbitrary number of elements. We denote the $n$-element trivial quandle by $\T_n$.

Note that the axioms (Q2) and (Q3) are equivalent to the map $S_x\colon Q \to Q$ given by $$S_x(y)=y*x$$ being an automorphism of $Q$ for each $x \in Q$. These automorphisms are called {\it inner automorphisms}, and the group generated by all such automorphisms is denoted by $\Inn(X)$. A quandle $X$ is called {\it involutary} if $S_x^2 = \id_Q$ for each $x \in Q$. For example, all  core quandles are involutary.

A subset of a quandle is called a {\it subquandle} if it is a quandle with respect to the underlying binary operation. Subracks are defined analogously. 
\medskip

\section{$Q$- and $G$-family of quandles and their  generalizations} \label{generalizations}

\subsection{$Q$-families of quandles}
The following definition can be found in
 \cite{I}. 

\begin{definition}[\cite{I}]
	Let $(Q, \circ)$ be a quandle. A $Q$-family of quandles is a non-empty set $X$ with a family of binary operations $*_a\colon X \times  X \to X, a \in  Q$,  satisfying the following axioms:
\begin{enumerate}
\item  for any $x \in  X$ and any $a \in Q$, $x *_a x = x$; 

\item for any $x \in X$ and any $a \in Q$, the map $S_{x, a}\colon X \to X$ defined by $S_{x,a}(y) = y*_a x$
is a bijection;

\item for any $x, y, z \in X$ and any $a, b \in Q$,
$$
(x *_{a} y) *_b z = (x *_b z) *_{a \circ b} (y *_b z).
$$
\end{enumerate}
\end{definition}

It follows from this definition that for any $a \in Q$ the set $X$ with the operation $*_a$ is a quandle.

Let $(Q, \circ)$ be a quandle, $(X, \{*_a \}_{a\in Q})$ be a $Q$-family of quandles. It can be routinely checked that binary operation 
$$
* \colon (X \times Q) \times (X \times Q) \to (X \times Q)
$$
defined by the rule
$$
(x, a) * (y, b) = (x *_b y, a \circ b)
$$
gives a quandle structure on the set $X \times Q$.

We construct a generalization of $Q$-family of quandles. In \cite{BarSin} some general construction of quandles was suggested. Let us recall it.

\begin{proposition}[\cite{BarSin}] \label{prop:gen_quandle}
Let $X$ and $S$ be two sets, $g\colon X \times X \to \Maps(S \times S, S)$ and $f\colon S \times S \to \Maps(X \times X, X)$ be two maps. Then the set $X \times S$ with the binary operation 
\begin{equation}\label{genralised-quandle-operation}
(x, s)* (y,t)= \big( f_{s, t}(x, y), ~g_{x, y}(s, t) \big)
\end{equation}
forms a quandle if and only if the following conditions hold:
\begin{enumerate}
\item $f_{s, s}(x, x)=x$ and $g_{x, x}(s, s)=s$ for all $x \in X$, $s \in S$;
\item  for each $(y, t) \in X \times S$, the map $(x, s) \mapsto  \big( f_{s, t}(x, y),~g_{x, y}(s, t) \big)$ is a bijection;
\item for all $x, y, z \in X$ and $s, t, u \in S$ hold 
\item[] $f_{g_{x, y}(s, t), u}\Big(f_{s, t}(x, y),~ z \Big)= f_{g_{x, z}(s, u), g_{y, z}(t, u)}\Big(f_{s, u}(x, z), ~f_{t, u}(y, z) \Big)$ 
\item[] and
\item[] $g_{f_{s, t}(x, y), z}\Big(g_{x, y}(s, t), ~u \Big)= g_{f_{s, u}(x, z), f_{t, u}(y, z)}\Big(g_{x, z}(s, u), ~g_{y, z}(t, u) \Big)$.
\end{enumerate}
\end{proposition}

Suppose that $S$ is a quandle $(Q, *)$ and $g\colon X \times X \to \Maps(Q \times Q, Q)$ is defined by the rule $g_{x, y}(s, t) = s * t$. In this case we get a construction of a quandle from \cite{A}:

\begin{proposition}[\cite{A}] \label{prop:spec_quandle}
Let $X$ be a set, $(Q, *)$ be a quandle,  and  $\mathfrak{f}\colon Q \times Q \to \Maps(X \times X, X)$ be a map. Then the set $X \times Q$ with the binary operation 
\begin{equation}\label{genralised-quandle-operation-1}
(x, s)\cdot (y,t)= \big( \mathfrak{f}_{s, t}(x, y), ~s * t \big)
\end{equation}
forms a quandle if and only if the following conditions hold:
\begin{enumerate}
\item $\mathfrak{f}_{s, s}(x, x)=x$  for all $x \in X$, $s \in Q$;
\item  for each $(y, t) \in X \times Q$, the map $(x, s) \mapsto  \big( \mathfrak{f}_{s, t}(x, y),~s *t \big)$ is a bijection;
\item for all $x, y, z \in X$ and $s, t, u \in S$ holds 
\item[] $\mathfrak{f}_{s * t, u}\Big(\mathfrak{f}_{s, t}(x, y),~ z \Big)= \mathfrak{f}_{s * u, t * u}\Big(\mathfrak{f}_{s, u}(x, z), ~\mathfrak{f}_{t, u}(y, z) \Big)$. 
\end{enumerate}
\end{proposition}

Underlying the quandle constructed in this proposition is actually a family of quandles generalizing the $Q$-family defined above.

\begin{definition}
	Let $(Q,*)$ be a quandle. A $(Q ,f)$-family of quandles is a non-empty set $X$ with a family of binary operations $*_a\colon X \times  X \to X, a \in  Q$,  satisfying the following axioms:
\begin{enumerate}
\item  for any $x \in  X$ and any $a \in Q$, $x *_a x = x$; 

\item for any $x \in X$ and any $a \in Q$, the map $S_{x, a}\colon X \to X$ defined by $S_{x,a}(y) = y*_a x$
is a bijection;

\item for any $x, y, z \in X$ and any $g, h, q \in Q$,
$$
(x*_{f(g, h)} y) *_{f(g*h, q)} z = (x*_{f(g, q)} z)  *_{f(g*q, h * q)}  (y *_{f(h, q)} z).
$$
\end{enumerate}
\end{definition}

It is easy to see that if we set $\mathfrak{f}(s,t)=*_{f(s,t)} : X \times X \to X$ for a given $(Q,f)$-family of quandles, then the set $X\times Q$ with the operation 
$$
(x,s)\cdot (y,t)=(x*_{f(s,t)} y, s*t)
$$
 is a quandle which coincides with the quandle defined in Proposition~\ref{prop:spec_quandle}. It is called the {\em associated quandle} of the $(Q,f)$-family of quandles. It is worth noting that the axioms of $(Q,f)$-family of quandles are a bit {\em stronger} than the ones needed for $(X\times Q,\cdot)$ to be a quandle. That is, first two axioms for the family must be satisfied for all $*_a, a\in Q$, while for $X\times Q$ to be a quandle we need those axioms to hold only for $*_{f(s,s)}, s\in Q$ (for the first axiom) and $*_{f(s,t)}, s,t\in Q$ (for the second axiom).

Now let us show that this construction generalizes the $Q$-family of quandles. Indeed, suppose that the map $f\colon Q \times Q \to \Maps(X \times X, X)$ depends only on the second argument, in other words 
$$
f_{s,t}(x, y) = f_{s',t}(x, y) \;\;\; \forall s,s'\in Q.
$$
If we put $f_{s,t}(x, y) = x*_t y$, then from Proposition~\ref{prop:spec_quandle} we get

\begin{corollary}
Let $X$ be a set, 
$(Q, \circ)$ be a quandle, and  $f\colon Q  \to \Maps(X \times X, X)$ a map, $f_{s,t}(x, y) = x*_t y$. Then the set $X \times Q$ with the binary operation 
\begin{equation}\label{cor}
(x, s)\cdot (y,t)= \big( x *_t y, ~s \circ t \big)
\end{equation}
forms a quandle if and only if the following conditions hold:
\begin{enumerate}
\item $x *_t x=x$  for all $x \in X$, $t \in Q$;
\item  for each $(y, t) \in X \times Q$, the map $(x, s) \mapsto  \big( x*_t y,~s \circ t \big)$ is a bijection;
\item for all $x, y, z \in X$ and $t, u \in Q$ holds 
\item[]  $(x*_t y) *_u z =  (x *_u z) *_{t \circ u} (y *_u z)$. 
\end{enumerate}
\end{corollary}

In particular, the $(Q,f)$-family of quandles is a $Q$-family of quandles if $*_{f(s,t)}:=*_t$.

\subsection{$G$-families of quandles}
If the set $Q$ is actually a group, a construction parallel (though not quite analogous) to $Q$-family of quandles may be defined and then generalized similar to the case of $(Q,f)$-family.

\begin{definition}[\cite{I}]
	Let $G$ be a group. A $G$-family of quandles is a non-empty set $X$ with a family of binary operations $*_g\colon X \times  X \to X$, $g \in  G$, satisfying the following axioms:
\begin{enumerate}
\item  for any $x \in  X$ and any $g \in G$, $x *_g x = x$; 

\item for any $x, y \in X$ and any $g, h \in G$,
$$
x *_{gh} y = (x *_g y) *_h y,~~~x *_e y = x;
$$

\item for any $x, y, z \in X$ and any $g, h \in G$,
$$
(x *_{g} y) *_h z = (x *_h z) *_{h^{-1}gh} (y *_h z).
$$
\end{enumerate}
\end{definition}

\begin{example}
	Let $H$ be a group, $G\subseteq \Aut H$ be a subgroup in its automorphisms group. For each $g\in G$ define $*_g\colon H\times H\to H, a*_g b = g(ab^{-1})b$, that is, for each $g$ define the structure of a generalized Alexander quandle on $H$. Then it is easy to check that $(H, \{*_g \}_{g \in G})$ is a $G$-family of quandles.
\end{example}

\begin{remark}
	This construction is in a way similar to the {\em holomorph} $\Hol H$ of $H$ construction. Recall, that a holomorph is the semi-direct product $\Hol H = H \rtimes \Aut H$ with the operation
$$
(x, \varphi)  (y, \psi)=(x^{\psi}  y, \varphi  \psi),~~x, y \in H, \varphi,  \psi \in \Aut H.
$$
If we define the operation $\cdot$ on $\Hol H$ by the rule
$$
(x, \varphi) \cdot  (y, \psi)=(x *_{\psi}  y, \varphi  \psi),~~x *_{\psi}  y = \psi(x y^{-1}) y,~~ x, y \in H, \varphi,  \psi \in \Aut H,
$$
we get a quandle.
\end{remark}

Generalizing the construction of $G$-family of quandles by introducing the map $f\colon G\times G\to \Maps(X \times X, X)$ we get the following

\begin{definition} \label{def:Gf-family}
	A $(G, f)$-{\it family of quandles} is a non-empty set $X$ with a family of binary operations $*_g\colon X \times  X \to X$, $g \in  G$, satisfying the following axioms:
\begin{enumerate}
\item  for any $x \in  X$ and any $g \in G$, $x *_g x = x$; 

\item for any $x, y \in X$ and any $g, h \in G$,
$$
x *_{gh} y = (x *_g y) *_h y,~~~x *_e y = x,
$$
where $gh$ is the product in $G$ and $e$ is the unit element of $G$;

\item for any $x, y, z \in X$ and any $g, h, q \in G$,
$$
 (x*_{f(g, h)} y) *_{f(g*h, q)} z = (x*_{f(g, q)} z)  *_{f(g*q, h * q)}  (y *_{f(h, q)} z).
$$ 
\end{enumerate}
\end{definition}

As in the case of $Q$-families of quandles, the set $X\times G$ is a quandle with the operation 
$$
(x,g)\cdot (y,h)=(x*_{f(g,h)} y, g*h),~~x, y \in X,  g, h \in G.
$$ 
It is also called the {\em associated quandle} of the family.

\begin{lemma} \label{lem:for}
Consider a $(G,f)$-family of quandles. For any $x, y \in X$ and any $g, h, q \in G$, the function $f$ satisfies the condition
$$
x*_{f(g, h)f(g*h,q)} y = x*_{f(g, q) f(g*q,h*q)} y.
$$ 
\end{lemma}

\begin{proof}
If we put $z=y$, then the third axioms of Definition \ref{def:Gf-family} has the form 
$$
 (x*_{f(g, h)} y) *_{f(g*h, q)} y = (x*_{f(g, q)} y)  *_{f(g*q, h * q)}  (y *_{f(h, q)} y).
$$
By the first axiom it is equivalent to 
$$
 (x*_{f(g, h)} y) *_{f(g*h, q)} y = (x*_{f(g, q)} y)  *_{f(g*q, h * q)} y.
$$
By the second axiom it is equivalent to 
$$
x*_{f(g, h) f(g*h, q)} y = x*_{f(g, q)f(g*q, h * q)} y.
$$
\end{proof}

\begin{remark}
Lemma~\ref{lem:for} means that the map $*_{f(g,h)}\colon G\times G\to {\rm Map} (X\times X\to X)$ satisfies the {\em 2-cocycle condition}, see, for example, \cite{BarSin}.
\end{remark}

Now let us present several examples of $(G,f)$-families of quandles.

\begin{example}
\label{ex:f-quandles}
1) If $Q_G = \Conj(G)$ is the conjugacy quandle that is a quandle with the operation $g * h = h^{-1} g h$,  $f(g, h) = h$ for all $g, h \in G$, then by Lemma~\ref{lem:for} we have
$$
x*_{hq} y = x*_{q q^{-1}  h q} y  \Leftrightarrow x*_{hq} y = x*_{ h q} y,
$$ 
and the $(G, f)$-family of quandles is the $G$-family of quandles.

2) If $Q_G = \T(G)$ is the trivial quandle that is a quandle with the operation $g * h = g$,  $f(g, h) = h$ for all $g, h \in G$, then by Lemma~\ref{lem:for} we have
$$
x*_{hq} y = x*_{q h} y.
$$ 
Hence, in this case the operations $*_g$ commute. For example, that occurs when $G$ is abelian.

3) If $Q_G = \Core(G)$ is the core quandle that is a quandle with the operation $g * h = h g^{-1} h$,  $f(g, h) = h$ for all $g, h \in G$, then by Lemma~\ref{lem:for} we have
$$
x*_{hq} y = x*_{q q h^{-1}q} y.
$$ 
It is true, for example, if $G$ has exponent 2.
\end{example}

\section{Multiplication of quandles}
\label{sec:quandle_mult}

\subsection{Definition of quandle multiplication}
Consider two quandles $Q_1 = (Q, \circ)$ and $Q_2 = (Q, *)$ defined on a set $Q$. Define the {\em composition} $\cs$ of operations $\circ$ and $*$ in the following way: $$\cs\colon Q\times Q\to Q; \;\;\; a\cs b=(a\circ b)*b.$$ 

A natural question arises: is $(Q,\cs)$ a quandle? In general, the answer is negative. 

\begin{example}
On the 3-element set $\{ 1, 2, 3 \}$ there exist three non-isomorphic quandles with multiplication tables:
$$
\begin{tabular}{|c||c|c|c|}
    \hline
$\T_3$ & 1 & 2 & 3  \\
  \hline \hline
1 & 1 & 1 & 1  \\
2 & 2 & 2 & 2  \\
3 & 3 & 3 & 3  \\
  \hline
\end{tabular},
\, \, \, \, \, \, \, \, \, \,
\begin{tabular}{|c||c|c|c|}
    \hline
$\R_3$ & 1 & 2 & 3  \\
  \hline \hline
1 & 1 & 3 & 2  \\
2 & 3 & 2 & 1  \\
3 & 2 & 1 & 3  \\
  \hline
\end{tabular},
\, \, \, \, \, \, \, \, \, \,
\begin{tabular}{|c||c|c|c|}
    \hline
$J_3$ & 1 & 2 & 3  \\
  \hline \hline
1 & 1 & 1 & 1  \\
2 & 3 & 2 & 2  \\
3 & 2 & 3 & 3  \\
  \hline
\end{tabular}.
$$
Here $\T_3$ is  the trivial quandle, $\R_3$ is the dihedral quandle and $J_3$ is the Joyce quandle. It is easy to see that 
$$
\T_3 \R_3 = \R_3 \T_3 = \R_3,~~~\T_3 J_3 =J_3 \T_3 = J_3,~~~\R_3 \R_3 = J_3 J_3 = \T_3 \T_3 = \T_3.
$$
On the other hand, the groupoids $\R_3 J_3$ and $J_3 \R_3$ have multiplication tables,
$$
\begin{tabular}{|c||c|c|c|}
    \hline
$\R_3 J_3$ & 1 & 2 & 3  \\
  \hline \hline
1 & 1 & 3 & 2  \\
2 & 2 & 2 & 1  \\
3 & 3 & 1 & 3  \\
  \hline
\end{tabular},
\, \, \, \, \, \, \, \, \, \,
\begin{tabular}{|c||c|c|c|}
    \hline
$J_3 \R_3$ & 1 & 2 & 3  \\
  \hline \hline
1 & 1 & 3 & 2  \\
2 & 2 & 2 & 1  \\
3 & 3 & 1 & 3  \\
  \hline
\end{tabular}.
$$
They are not quandles but are idempotent right quasigroups.

\end{example}

In the general case the following holds:

\begin{lemma}\label{lem:cs_quandle} Let $Q$ be a set and let $\circ$ and $*$ be two quandle operations on the set $Q$. Then, 

1) The composition $\cs$ satisfies the idempotency and right invertibility axioms;

2) If the operation $*$ is distributive with respect to the operation $\circ$, i.e.
$$
(a \circ b) * c = (a * c) \circ (b * c),~~~a, b, c \in Q,
$$
then the composition $\cs$ satisfies the self-distributivity axiom.	
\end{lemma} 

\begin{proof}
	1) {\bf Idempotency}: $a\cs a = (a\circ a)*a = a*a = a$ since both operations are idempotent.
	
	{\bf Right invertibility}: for the equation $y=x\cs a$ we have: $$y=(x\circ a)*a,$$ therefore, $x=(y\bs a)\bc a$, where $\bs$ and $\bc$ denote the right inverse operations for $*$ and $\circ$, respectively.
	
	2) {\bf Self-distributivity}: we need to prove that for any $a,b,c\in Q$ we have $$(a\cs b)\cs c=(a\cs c)\cs (b\cs c).$$ By definition and taking into account the distributivity of $*$ with respect to $\circ$, we get: $$(a\cs b)\cs c= (((a\circ b)*b)\circ c)*c=(((a\circ b)\circ c)*(b\circ c))*c=$$ $$=(((a\circ b)\circ c)*c)*((b\circ c)*c)=(((a\circ c) \circ (b\circ c))*c)*((b\circ c)*c)=$$ $$=(((a\circ c)*c)\circ ((b\circ c)*c))*((b\circ c)*c)=(a\cs c)\cs (b\cs c).$$
\end{proof}

\begin{remark}
From the proof of this lemma we can see that for the operation $\cs$, the operation $\bs\bc$ is the right inverse.	
\end{remark}

\begin{remark}
	This lemma gives a {\em sufficient} condition for $\cs$ to be a quandle operation. Below we show that this condition is not {\em necessary}.
\end{remark}

This lemma allows one to produce a new quandle structure from two given ones, provided they satisfy the distributivity condition.

\begin{definition} \label{def:quandle_mult}
	Let $Q$ be a set and let $\circ$ and $*$ be two quandle operations on $Q$ which satisfy  the condition 
	$$
	(a\circ b)*c=(a*c)\circ (b*c)
	$$ for any $a,b,c \in Q$. Then {\em multiplication} of quandles $(Q,\circ)$ and $(Q,*)$ is defined by the following formula: 
	\begin{equation}
		(Q,\circ) (Q,*)=(Q,\cs).
	\end{equation}
\end{definition}

As follows from Lemma \ref{lem:cs_quandle}, $(Q,\cs)$ is a quandle. We will denote the groupoid $(Q,\cs)$ obtained by multiplication of $Q_1=(Q,\circ)$ and $Q_2=(Q,*)$ by $Q_1 Q_2$.

\begin{example}
	Let $*=\circ$. Quandle opearation is self-distributive, therefore the operation $*^2=**$ is a quandle operation as well. That is, given a quandle $(Q,*)$ we define its square as the quandle $(Q,*^2)$. In this language an {\em involutive} quandle may be defined as a quandle, whose square is the trivial quandle.
\end{example}

To use the multiplication operation further, we need to understand how powers of quandle operations interact with one another. First, let us introduce the formal definition.

\begin{definition}
	Let $(Q,*)$ be a quandle. For any $n\in \mathbb{N}$ and any $a,b\in Q$ define:
	\begin{itemize}
	\item $a*^0 b = a$;
	\item $a*^n b = (a*^{n-1} b)*b$;
	\item $a*^{-n}b = a \bs^n b$.
	\end{itemize}
\end{definition}

\begin{lemma} \label{lem:powers_distr}
	Let $(Q,\circ)$ and $(Q,*)$ be two quandles, and let $*$ be distributive with respect to $\circ$. Then for any $n,m\in \mathbb{Z}$ the operation $*^m$ is distributive with respect to $\circ^n$.
\end{lemma}

\begin{proof}
	We shall prove the claim by going over different values of $m$ and $n$ and using induction where possible.
	
	1) If $n=0$ we have $$(a\circ^0 b)*^m c = a*^m c = (a*^m c)\circ^0 (b*^m c).$$ Likewise, if $m=0$, we have $$(a\circ^n b)*^0 c = a\circ^n b = (a*^0 c) \circ^n (b*^0 c).$$
	
	2) If $m=n=1$, the claim is true from the assumption.
	
	3) For the general subcase $m,n\ge 0$ let us use the induction by $m+n$. Induction basis $m+n=1$ was proved above. Suppose the claim is true for $m+n\le k$. Consider $m+n=k+1$. We shall study the following subcases:
	\begin{itemize}
	\item{$m=1, n>1$.} In that case we have $(a\circ^n b)*c = ((a\circ^{n-1} b)\circ b)*c=((a\circ^{n-1} b)*c)\circ (b*c)= ((a* c)\circ^{n-1}(b*c))\circ (b*c) = (a*c)\circ^n (b*c).$	
	\item{$m>1$.} In this case, $(a\circ^n b)*^mc = ((a\circ^n b)* c)*^{m-1}c=((a*c)\circ^n (b*c))*^{m-1}c=(a*^m c)\circ^n (b*^m c).$
	\end{itemize}
	
	Now we need to move on to the case of negative powers. Essentially, we just need to understand how the operations $\bc$ and $\bs$ behave.
	
	4) If $n=1, m=-1$, we need to check that 
$(a\circ b)\bs c = (a\bs c)\circ (b\bs c)$. We have $a\circ b = ((a\bs c)*c)\circ ((b\bs c)*c) = ((a\bs c)\circ (b\bs c))* c$. By applying $\bs c$ to both sides of the equality, we get the desired distributivity.
	
	5) If $n=-1, m=-1$, let us begin with the equality $x \circ b = a$, which is equivalent to $x = a \bc b$. That gives us $(x\circ b)\bs c = a\bs c$. By case 4) that yields $(x\bs c)\circ(b\bs c)=a\bs c$. Hence $x\bs c = (a\bs c)\bc (b\bs c)$. Finally, by $x = a \bc b$ we have the desired $(a \bc b)\bs c = (a\bs c)\bc (b\bs c)$.
	
	6) Now we can deal with the case $n=-1, m=1$. In the same manner as in case 4) and using case 5) we have $a\bc b = ((a*c)\bs c)\bc ((b*c)\bs c)=((a*c)\bc(b*c))\bs c$, therefore $(a\bc b)*c = (a*c)\bc (b*c)$.
	
	7) If $n,m<0$ we use induction by $|n|+|m|$ repeating case 3) verbatim for operations $\bc$ and $\bs$. Necessary preliminary cases of $|n|,|m|\le 1$ were already proved in 4), 5) and 6).
	
	8) The last case is the mixed one: when $n$ and $m$ have different signs. But cases 4), 5) and 6) together with the assumption of the lemma give us that $*^{\pm 1}$ is distributive with respect to $\circ^{\pm 1}$. Therefore, the same inductive technique applied to suitable operations ($\bc$ and $*$ or $\circ$ and $\bs$, depending on the signs of $n$ and $m$) gives us the desired result. 
\end{proof}

\begin{example} \label{ex:quandle_power}
	Consider a quandle $(Q,*)$. Due to Lemma~\ref{lem:powers_distr} we see that for any $n,m\in\mathbb{Z}$ the powers $*^n$ and $*^m$ are distributive with respect to each other. Therefore, for any $k\in\mathbb{Z}$ the operation $*^k$ is quandle, since it can be obtained as a composition of quandle operations, distributive with respect to one another.
	
	That means that we have correctly defined an {\em integer power of a quandle $(Q,*)$}: it is the quandle $(Q,*^k), k\in\mathbb{Z}$. In particular, 0-th power of any quandle is the trivial quandle.
\end{example}

The situation when $*$ is distributive with respect to $\circ$ leads to an interesting consequence. We get the following chain of equalities:

$$a\cs b=(a\circ b)*b=(a*b)\circ(b*b)=(a*b)\circ b = a\cc b.$$

Hence we proved that $\cs = \cc$. In particular, $\cc$ is also a quandle operation. This fact actually allows us to answer the question: is distributivity $*$ with respect to $\circ$ the {\em necessary} condition for $\cs$ to be a quandle operation? Consider the following example.

\begin{example}
\label{ex:core_and_conj}
	Let $G$ be a non-abelian group, and let $(G,\circ)=Conj(G), (G,*)=Core(G)$. Let us check the distributivities.
	
	To begin with, $$(a*b)\circ c=c^{-1}ba^{-1}bc=c^{-1}bcc^{-1}a^{-1}cc^{-1}bc=(a\circ c)*(b\circ c),$$ hence $\circ$ is distributive with respect to $*$, and the operation $\cc$ is quandle. Therefore, the operation $\cs$ is quandle as well. At the same time,
	
	$$(a*c)\circ(b*c)=c^{-1}bc^{-1}ca^{-1}ccb^{-1}c=c^{-1}ba^{-1}c^2b^{-1}c\neq cb^{-1}a^{-1}bc=(a\circ b)*c.$$
	That shows that $*$ is not distrivutive with respect to $\circ$, and this condition is not necessary for $\cs$ to be a quandle operation.
\end{example}

\subsection{A group of quandles closed under quandle multiplication}
If we begin with operations $\circ$ and $*$ and consider iterative compositions of their powers, we get an infinite family of operations, each of which is defined by a word in the alphabet $\{\circ, *, \bc, \bs\}$. For example, the word $\circ\circ\bs\circ *$ defines the operation $(a,b)\mapsto ((a\circ^2 b)\bs b)\circ b$. We can say that each operation of this family is defined by an element of the free group $F_2=\langle\circ,*\rangle$.

Example~\ref{ex:quandle_power} gives us hope that if we impose some conditions on the operations $\circ$ and $*$, iterative compositions would always give us quandle operations. It is actually the case.

\begin{theorem} \label{thm:quandle_mult}
	Let $(Q,\circ)$ and $(Q,*)$ be quandles and let the operations $\circ$ and $*$ be distributive with respect to each other, i.e.
$$
(a \circ b) * c = (a * c) \circ (b * c),~~~(a * b) \circ c = (a \circ c) * (b \circ c),~~~a, b, c \in Q.
$$
Then any finite word in the alphabet $\{\circ, *, \bc, \bs\}$ defines a quandle operation on the set $Q$.
\end{theorem} 

\begin{proof}
	By Lemma~\ref{lem:cs_quandle} and Lemma~\ref{lem:powers_distr} we see that for any $n,m\in\mathbb{Z}$ the operation $\circ^n *^m$ is a quandle operation. That means that any word of syllable length not greater than 2 defines a quandle operation on $Q$. Let us consider compositions of such operations. Composition of quandle operations always satisfies the idempotency and right invertibility axioms. Therefore we only need to check the self-distributivity. Hence, we need to check that operations of the form $\circ^n *^m$ are distributive with respect to one another for any $n,m\in\mathbb{Z}$. The direct check gives the following: $$(a\circ^n*^m b)\circ^k*^l c = (((a\circ^n b)*^m b)\circ^k c)*^l c=(((a\circ^n b)\circ^k c)*^m(b\circ^k c))*^l c=$$ $$=(((a\circ^n b)\circ^k c)*^l c)*^m(b\circ^k *^l c)=(((a\circ^k c)\circ^n(b\circ^k c))*^l c)*^m (b\circ^k *^l c)= $$ $$= ((a\circ^k *^l c)\circ^n (b\circ^k *^l c))*^m (b\circ^k*^l c)=(a\circ^k*^l c)\circ^n*^m (b\circ^k*^l c).$$ This sequence of equalities proves the distributivity of $\circ^k*^l$ with respect to $\circ^n*^m$, and therefore all operations defined by the words of syllable length not greater than 4 are self-distributive (and hence are quandle operations).
	
	Note that in this argument we used both distributivity of $*$ with respect to $\circ$ and vice versa. This condition is natural, because setting $n=l=0$ we exactly get that operations of the form $*^m\circ^k$ are quandle, which suggests that $\circ$ must be distributive with respect to $*$. 
	
	Now we need to pass to words of syllable length not greater than 8 by studying the compositions of operations given by shorter words. That may be done technically in the same manner as we did for words of syllable length 4. But there is a less technical way to prove it.
	
	Any operation given by the word $w$ of syllable length 8 or less is a composition of two words $w'_4$ and $w''_4$ of syllable length 4 or less. They, in turn, are composed of two words of syllable length 2 or less each, that is of operations of the form $\circ^n *^m$. We have proved that those operations are distributive with respect to each other. That means that the reasoning used to prove their distributivity may be applied to them (instead of the initial operations $\circ$ and $*$) to prove that length 4 operations are distributive, and hence any length 8 operation is self-distributive. 

Continuing this process inductively, we get the desired result: any operation defined by a finite word in the alphabet $\{\circ, *, \bc, \bs\}$ is self-distributive, and hence is a quandle operation.
\end{proof}

Theorem~\ref{thm:quandle_mult} may be interpreted in the following way. Consider a set $Q$ and two quandle operations on it: $\circ$ and $*$. Let these operations be distributive with respect to each other. Then the closure of the set $\{(Q,\circ), (Q,*)\}$ with multiplication defined in Definition~\ref{def:quandle_mult} is a set of quandles. In other words, if we denote by $\star_w$ an operation defined by a word $w$ from the free group $F_2=\langle\circ, *\rangle$, then the set $\{(Q,\star_w)\,|\,w\in F_2\}$ is closed under quandle multiplication.

Example~\ref{ex:quandle_power} is a natural partial case of this construction when $\circ=*$.

Naturally, this construction may be extended to the case of more operations. To be precise, the following theorem holds.

\begin{theorem} \label{gent}
	Let $Q$ be a set and $*_\alpha$ be a set of quandle operations on $Q$ indexed by some set $\Lambda$. Let the operations $*_\alpha$ be pairwise distributive with respect to each other. Furthermore, let $F$ be the free group generated by $*_\alpha$ for all $\alpha\in \Lambda$. Let $\star_w\colon Q\times Q\to Q$ be an operation defined by the word $w\in F$ as described above. Then the set $\{(Q,\star_w)\,|\,w\in F\}$ is a set of quandles closed under the quandle multiplication.
\end{theorem}

Moreover, quandle multiplication is associative (even though quandle operations themselves generally are not), and hence the set $\{(Q,\star_w)\,|\,w\in F\}$ has a natural group structure: the operation is the quandle multiplication, unit element is given by the trivial quandle, corresponding to the empty word, and the inverse element for $(Q,\star_w)$ is the element $(Q,\star_{w^{-1}})$. Let us denote this group by $Q_F$.

\begin{remark}
	The group structure is introduced on the family of quandles themselves, not their isomorphism classes. For example, for any quandle $(Q,\circ)$ we have an inverse element which actually is the quandle $(Q,\bar\circ)$. Those two quandles are of course isomorphic. On the other hand, if we multiple the quandle $(Q,\circ)$ by itself, we have no reason to expect the result to be the trivial quandle. Hence we can't say that the group structure respects the isomorphism relation on quandles.

\end{remark}

We can study the group $Q_F$ a bit further. 

\begin{proposition}
Let $F=F_n$ be the free group on $n$ generators. Then the group $Q_F$ has a structure of a $(G,f)$-family of quandles.	
\end{proposition}

\begin{proof}

Let us consider the case of $n=2$: $F=F_2=\langle\circ, \star \rangle$, that is, the set of operations is generated by two mutually distributive operations $\circ, \star$. The case of arbitrary $n$ is considered in the same manner.

To prove the statement of the proposition, we first need to present the mapping $f$ and the operation $*$. Let $*$ be the trivial quandle operation on $F_2$, and let us consider the mapping $f\colon F\times F \to F$ defined as $f(g,h)=h$ for all $g,h\in F$, that is, $f$ is the projection on the second argument. 

Consider the family $(Q; \star_w, w\in F_2; f)$ with the trivial quandle operation on $F_2$. First, we note that $x \star_w x = x$ for all $w\in F_2$ and $x \star_e y = x$ for all $x,y\in Q$ where $e$ is the trivial element of $F_2$ (the empty word). Furthermore, $(x\star_w y)\star_{w'} y = x \star_{ww'} y$ by definition of the operations $\star_w$. Finally, 
$$
(x\star_{f(g,h)} y)\star_{f(g*h,q)}z = (x\star_h y)\star_q z = (x\star_q z)\star_h (y \star_q z) = (x\star_{f(g,q)}z)\star_{f(g*q,h*q)}(y\star_{f(h,q)} z)
$$ due to operations $\star_h, \star_q$ being mutually distributive for all $h,q\in F_2$.

All axioms of a $(G,f)$-family are satisfied, and hence the claim is proved.

\end{proof}


This proposition gives us the following corollary. As it was shown in Example~\ref{ex:f-quandles}(2), this family is commutative in the sense that $*_{gh}=*_{hg}$ for all $g,h \in F$. Hence, the group $Q_F$ is commutative, and really is generated by quandles $(Q,*_g)$ where $g\in \mathbb{Z}^2 \cong F/[F,F]$, that is $g=\circ^n *^k, \, n,k\in\mathbb{Z}$.


\begin{corollary} \label{abel}
	The group $Q_F$ of quandles defined above is commutative.
\end{corollary}

\subsection{Further examples of quandle multiplication}
Now let us present some other examples of quandle multiplication. 

1. Recall (see \cite{J}) that a quandle $(Q, *)$ is called {\em an $n$-quandle} if for any $a, b \in Q$,
$$
(\ldots ((a * b) * b) * \ldots ) * b = a,
$$
where $b$ appears $n$ times in the left-hand side of the equality. In particular, $1$-quandle is the trivial quandle, $2$-quandle is an involutory quandle. Some examples of $n$-quandles are well-known. For example, any core quandle is involutory, and hence a 2-quandle. Moreover, consider an Alexander quandle $(A, \varphi)$, where $A$ is abelian group, $\varphi$ is its automorphism, and the operation is defined by the rule 
$$
a * b = \varphi(a-b) + b, ~~a, b \in A.
$$ 
If $\varphi$ has order $n$, it is an $n$-quandle. 

In our definition, if $(Q, *)$ is an $n$-quandle, then the set of quandles
$$
\{(Q, *),~~(Q, *^2),~~\ldots,~~(Q, *^n)\}
$$
forms a cyclic group of order $n$ under quandle multiplication and $(Q, *^n)$ is the trivial quandle.

\begin{example}
Let $A = \{ 0, 1, 2, 3, 4 \}$ be the cyclic group of order 5 and $\varphi \in \Aut(A)$, $\varphi(a) = 2a$, $a \in A$. Then $\varphi^4 = id$ and we get 3 non-trivial quandles
$$
\begin{tabular}{|c||c|c|c|c|c|}
    \hline
$*$ & 0 & 1 & 2 & 3 & 4  \\
  \hline \hline
0 & 0 & 4 & 3  & 2 & 1\\
1 & 2 &1 & 0  & 4 & 3\\
2 & 4 & 3 & 2  & 1 & 0\\
3 & 1 & 0 & 4 & 3 & 2  \\
4 & 3 & 2 & 1 & 0 & 4\\
  \hline
\end{tabular},
\, \, \, \, \, \, \, \, \, \,
\begin{tabular}{|c||c|c|c|c|c|}
    \hline
$*^2$ & 0 & 1 & 2 & 3 & 4  \\
  \hline \hline
0 & 0 & 2 & 4  & 1 & 3\\
1 & 4 &1 & 3  & 0 & 2\\
2 & 3 & 0 & 2  & 4 & 1\\
3 & 2 & 4 & 1 & 3 & 0  \\
4 & 1 & 3 & 0 & 2 & 4\\
  \hline
\end{tabular},
\, \, \, \, \, \, \, \, \, \,
\begin{tabular}{|c||c|c|c|c|c|}
    \hline
$*^3$ & 0 & 1 & 2 & 3 & 4  \\
  \hline \hline
0 & 0 & 3 & 1  & 4 & 2\\
1 & 3 &1 & 4  & 2 & 0\\
2 & 1 & 4 & 2  & 0 & 3\\
3 & 4 & 2 & 0 & 3 & 1  \\
4 & 2 & 0 & 3 & 1 & 4\\
  \hline
\end{tabular}.
$$
It is easy to see that $(A, *^2) = \R_5$ is the dihedral quandle.
\end{example}

2. Quandles of order 4. In the paper \cite{HN} one can find 7 non-isomorphic quandles of order 4. They have multiplication tables
$$
\begin{tabular}{|c||c|c|c|c|}
    \hline
$*_0$ & 0 & 1 & 2 & 3   \\
  \hline \hline
0 & 0 & 0 &  0 & 0\\
1 & 1 &1 & 1  & 1 \\
2 & 2 & 2 & 2  & 2 \\
3 & 3 & 3 & 3 & 3   \\
  \hline
\end{tabular},
\, \, \, \, \, \, \, \, \, \,
\begin{tabular}{|c||c|c|c|c|}
    \hline
$*_1$ & 0 & 1 & 2 & 3   \\
  \hline \hline
0 & 0 & 0 & 0  & 0\\
1 & 1 &1 & 1  & 2 \\
2 & 2 & 2 & 2  & 1 \\
3 & 3 & 3 & 3 & 3   \\
  \hline
\end{tabular},
\, \, \, \, \, \, \, \, \, \,
\begin{tabular}{|c||c|c|c|c|}
    \hline
$*_2$ & 0 & 1 & 2 & 3   \\
  \hline \hline
0 & 0 & 0 & 0  & 1\\
1 & 1 &1 & 1  & 2 \\
2 & 2 & 2 & 2  & 0 \\
3 & 3 & 3 & 3 & 3   \\
  \hline
\end{tabular},
\, \, \, \, \, \, \, \, \, \,
\begin{tabular}{|c||c|c|c|c|}
    \hline
$*_3$ & 0 & 1 & 2 & 3   \\
  \hline \hline
0 & 0 & 0 & 1  & 1\\
1 & 1 &1 & 0  & 0 \\
2 & 2 & 2 & 2  & 2 \\
3 & 3 & 3 & 3 & 3   \\
  \hline
\end{tabular},
$$
$$
\begin{tabular}{|c||c|c|c|c|}
    \hline
$*_4$ & 0 & 1 & 2 & 3   \\
  \hline \hline
0 & 0 & 0 & 0  & 0\\
1 & 1 &1 & 3  & 2 \\
2 & 2 & 3 & 2  & 1 \\
3 & 3 & 2 & 1 & 3   \\
  \hline
\end{tabular},
\, \, \, \, \, \, \, \, \, \,
\begin{tabular}{|c||c|c|c|c|}
    \hline
$*_5$ & 0 & 1 & 2 & 3   \\
  \hline \hline
0 & 0 & 0 & 1  & 1\\
1 & 1 &1 & 0  & 0 \\
2 & 3 & 3 & 2  & 2 \\
3 & 2 & 2 & 3 & 3   \\
  \hline
\end{tabular},
\, \, \, \, \, \, \, \, \, \,
\begin{tabular}{|c||c|c|c|c|}
    \hline
$*_6$ & 0 & 1 & 2 & 3   \\
  \hline \hline
0 & 0 & 3 & 1  & 2\\
1 & 2 &1 & 3  & 0 \\
2 & 3 & 0 & 2  & 1 \\
3 & 1 & 2 & 0 & 3   \\
  \hline
\end{tabular}.
$$
We will denote these quandles by $Q_i$, $i = 0, 1, \ldots, 6$, respectively. Then $Q_0$ is the trivial quandle, $Q_5$ is the dihedral quandle $\R_4$. It is easy to check that $Q_1^2 = Q_3^2 = Q_4^2 = Q_5^2 = Q_0$. Further,  $Q_2^3 = Q_6^3 = Q_0$, and since $*_2^2 = \overline{*_2}$, $*_6^2 = \overline{*_6}$, where $\overline{*_i}$
is the inverse operation to $*_i$, then $Q_2^2 \cong Q_2$, $Q_6^2 \cong Q_6$.

It is interesting to understand: for what $i, j \in \{ 1, 2, \ldots, 6\}$, $i \not= j$, the algebraic system $Q_i Q_j$ is a quandle?

\medskip

3. Let us return to the quandles from Example~\ref{ex:core_and_conj}: suppose that $G$ is a non-abelian group and define two quandles: the conjugacy quandle  $\Conj(G) = (G, *)$ with the operation $g * h = h^{-1} g h$
and the core quandle  $\Core(G) = (G, \circ)$ with the operation $g \circ h = h g^{-1}  h$. 

\begin{proposition}
In  $G$ the identity 
\begin{equation} \label{eq1}
(a \circ b) * c = (a  * c) \circ (b * c)
\end{equation}
holds for any $a, b, c \in G$. If $g^2$ lies in the center of $G$ for any $g \in G$, then the identity 
\begin{equation} \label{eq2}
(a * b) \circ c = (a  \circ c)  * (b \circ c)
\end{equation}
holds for any $a, b, c \in G$.
\end{proposition}

\begin{proof}
The left side of (\ref{eq1}) has the form
$$
(a \circ b) * c = (b a^{-1} b) * c = c^{-1} b a^{-1} b  c.
$$
The right side of (\ref{eq1}) has the form
$$
(a  * c) \circ (b * c) = (c^{-1} a c) \circ (c^{-1} b  c) = c^{-1} b a^{-1} b  c.
$$
Hence, the identity (\ref{eq1}) holds.

Further, the left side of (\ref{eq2}) has the form
$$
(a * b) \circ c = c b^{-1} a^{-1} b c.
$$
The right side of (\ref{eq2}) has the form
$$
(a  \circ c)  * (b \circ c) = (c a^{-1} c) * (c b^{-1} c) = c^{-1} b a^{-1} c^2 b^{-1} c.
$$
If $c^2 \in Z(G)$, then 
$$
c^{-1} b a^{-1} c^2 b^{-1} c = c b^{-1} a^{-1} b c.
$$
Hence, in this case (\ref{eq2}) holds.
\end{proof}

If  the square of any element of $G$ lies in the centre $Z(G)$, then $G/Z(G)$ has exponent 2, i.e. it is abelian, and $G$ is a two-step nilpotent group. Also, if we denote
$G^* = \Conj(G)$ and $G^{\circ} = \Core(G)$, then $G^{\circ} G^{\circ} = \T(G)$ is the trivial quandle on the set $G$. Also, from the equality 
$$
a ** b = (a*b)*b = b^{-2} a b^2 = a
$$
it follows that
$$
G^* G^*  = \T(G).
$$
Since
$$
a \circ * b = a * \circ b,~~a , b \in G,
$$
we get $G^{\circ} G^* = G^* G^{\circ}$ and $(G^{\circ} G^*)^2 = \T(G)$.

\begin{corollary} \label{conjcor}
If $G$ is a two-step nilpotent group in which the square of any element lies in the center, then $\Conj(G)$ and $\Core(G)$ generate a group that is isomorphic to abelianization
$G^{ab} = G / G'$.
\end{corollary}

\medskip

4. Let us consider generalized Alexander quandles and find conditions under which the composition of operations defines the multiplication of quandles.
Suppose that $G$ is a group, $\varphi, \psi \in \Aut(G)$. We can define two quandles on $G$. The first one is a quandle $G^*$ with multiplication
$a * b = \varphi(a b^{-1}) b$, $a, b \in G$; the second one is a quandle $G^{\circ}$ with multiplication
$a \circ b = \psi(a b^{-1}) b$, $a, b \in G$.

\begin{proposition} \label{alex}
If $\varphi \psi = \psi \varphi$, then for any integers $k, l$ the groupoid $(G, *^k \circ^l)$ is a quandle.
\end{proposition}

\begin{proof}
Let us find conditions under which the identity
$$
(a * b) \circ c = (a  \circ c)  * (b \circ c)
$$
holds for all $a, b, c \in G$. This identity is equivalent to
$$
\psi(\varphi(a b^{-1}) b c^{-1}) = \varphi(\psi(a b^{-1}) \psi(b c^{-1})). 
$$
Since $\psi$ is an automorphism, then $\psi(\varphi(a b^{-1}))  = \varphi(\psi(a b^{-1}))$, i.e. $\psi$ and $\varphi$ are permutable. 

By analogy, the identity 
$$
(a \circ b) * c = (a  * c)  \circ (b * c)
$$
holds for all $a, b, c \in G$ if $\varphi \psi = \psi \varphi$. The claim now follows from Theorem \ref{gent}.
\end{proof}

\bigskip

\section{Questions for further research} \label{ques}

\begin{question}
Let $\mathcal{Q}_n$ be the set of all $n$-element quandles. What maximal subgroups can be defined on this set?  
\end{question}

Note that there exist 3 non-isomorphic 3-element quandles, 7 non-isomorphic 4-element quandles, 22 non-isomorphic 5-element quandles, 73 non-isomorphic 6-element quandles, 298 non-isomorphic 7-element quandles. As we have seen in the present paper, on $\mathcal{Q}_3$ one can construct two maximal subgroups which are cyclic groups of order two.

\begin{question}
 Is there a $(Q,  f)$-family of quandles $(X, \{*_g\}_{g \in Q})$ such that any  $n$-element quandle is isomorphic to $(X, *_g)$ for some $g \in Q$? 
\end{question}

As in the case of groups we can define the rank of a quandle  $X$ as a minimal number elements of $X$ which generate $X$. The rank of $X$ is denoted by $rk(X)$.
It is easy to see that $rk(\Core(\mathbb{Z}_n)) = rk(\Core(\mathbb{Z})) = 2$ for $n > 1$.
We can formulate

\begin{question}
What is the connection between the rank of a group $G$ and the rank of $\Core(G)$? Is it true that $rk(\Core(G)) = rk(G) + 1$?
\end{question}

\begin{question}
	As we have shown in Example~\ref{ex:core_and_conj}, the distributivity is not a necessary condition for $\cs$ to be a quandle operation. How one can describe the actual necessary conditions on $\circ$ and $*$ and criteria for their composition to give a quandle operation?
\end{question}

\begin{question}
	In Corollary~\ref{conjcor} we have constructed a group of quandles, and hence a $(G,f)$-family of quandles. Is it possible to describe the associated quandle defined by this family?
\end{question}

\begin{ack}
This work was supported by the Ministry of Science and Higher Education of Russia (agreement No. 075-02-2021-1392). Also, the authors thank Sergey Shpectorov  for useful discussion.
\end{ack}

\end{document}